\documentclass[onecolumn,12pt]{IEEEtran} %!PN
\usepackage{amssymb}
\usepackage{amsthm,color}
\usepackage{latexsym}
\usepackage{graphicx}
\usepackage{fancyhdr}
\usepackage{bbm}
\usepackage{amsmath}
\usepackage{latexsym,amssymb,amsxtra,mathrsfs,times}
\newtheorem{thm}{Theorem}

\newtheorem{lemma}[thm]{Lemma}

\theoremstyle{definition}

\newtheorem{exam}[thm]{Example}

\newcommand{\lf}[2]{\left(\frac{#1}{#2}\right)}
\newcommand{\tr}{{\mathrm{Tr}}}
\newcommand{\gf}{{\mathbb F}}

\begin{document}

\title{Weight Distributions of a Class of Cyclic Codes with Arbitrary Number of Zeros II}

\author{
Jing Yang,\thanks{J. Yang is at the Department of Mathematical Sciences, Tsinghua University,
Beijing, 100084, China (email: jingyang@math.tsinghua.edu.cn).}
Lingli Xia, \thanks{L. Xia is at the Basic Courses Department of Beijing Union University, Beijing, 100101, China \& the Department of Mathematical Sciences, Tsinghua University,
Beijing, 100084, China (email: lingli@buu.edu.cn).}
Maosheng Xiong \thanks{M. Xiong is at the Department of Mathematics, The Hong Kong University of Science and Technology, Clear Water Bay, Kowloon, Hong Kong (email: mamsxiong@ust.hk).}
}
\maketitle

%% Classification and key words; note that the 2010 classification is used:

\renewcommand{\thefootnote}{}

%\footnote{2000 \emph{Mathematics Subject Classification}: 11L05;11L20.}

%\renewcommand{\thefootnote}{\arabic{footnote}}
%\setcounter{footnote}{0}

%%%%%%%%

\begin{abstract}
Cyclic codes are an important class of linear codes, whose weight distribution have been extensively studied. So far, most of previous results obtained were for cyclic codes with no more than three zeros. Recently, \cite{Y-X-D12} constructed a class of cyclic codes with arbitrary number of zeros, and computed the weight distributions for several cases. In this paper, we determine the weight distribution for a new family of such codes. This is achieved by certain new methods, such as the theory of Jacobi sums over finite fields and subtle treatment of some complicated combinatorial identities.

%In this paper we determine the weight distribution for a new family of cyclic codes that may have arbitrary number of zeros. This is achieved by exploring the construction of \cite{Y-X-D12} much further, and by theory of Jacobi sums over finite fields and subtle treatment of complicated combinatorial identities. Except for the recent works \cite{gegeng,Y-X-D12}, most of previous results obtained so far are for cyclic codes with no more than three zeros.
\end{abstract}


\section{Introduction}\label{sec-into}

A linear code $\mathcal{C}$ over the finite field $\mathbb{F}_q$ of length $n$ is a subspace of $\mathbb{F}_{q}^n$. It is called \textit{cyclic} if it also satisfies that any $(c_0,c_1,\cdots ,c_{n-1})\in \mathcal{C}$ implies $(c_{n-1},c_0,\cdots,c_{n-2})\in \mathcal{C}$. By the one-to-one correspondence
$$\begin{array}{cccl}
\sigma:& \mathcal{C}&\rightarrow &R:=\mathbb{F}_{q}[x]/(x^n-1)\\
 &(c_0,c_1,\cdots ,c_{n-1})&\mapsto&c_0+c_1x+\cdots +c_{n-1}x^{n-1},
\end{array}$$
each cyclic code $\mathcal{C}$ is equivalent to an ideal of $R$. Since $R$ is a principal ideal ring, there exists a unique monic polynomial $g(x)$ with least degree such that $\sigma(\mathcal{C})=g(x)R$ and $g(x)\mid (x^n-1)$. The $g(x)$ is called the \textit{generator polynomial} of $\mathcal{C}$ and $h(x):=(x^n-1)/g(x)$ is called the \textit{parity-check polynomial} of $\mathcal{C}$. The cyclic code $\mathcal{C}$ is called irreducible (resp. reducible) if $h(x)$ is irreducible (resp. reducible) over $\mathbb{F}_q$. For $\mathcal{C}$ reducible, we say that $\mathcal{C}$ \textit{has $t$ ($\ge 2)$ zeros} if $h(x)$ has $t$ irreducible factors over $\mathbb{F}_{q}$. (In the literature some authors call $\mathcal{C}$ as ``the dual of a cyclic code with $t$ zeros'' instead.)

Denote by $A_i$ the number of codewords of $\mathcal{C}$ with Hamming weight $i$. The {\em weight enumerator} of $\mathcal{C}$ with length $n$ is a polynomial in $\mathbb{Z}[Y]$ defined by
$$A_0+A_1Y+A_2Y^2+ \cdots + A_nY^n.$$
The sequence $(A_0,A_1,\cdots ,A_n)$ is called the \textit{weight distribution} of $\mathcal{C}$. The study of weight distribution of a linear code is important in both theory and application, since the weight distribution of a code gives the minimum distance and thus the error correcting capability of the code, and the weight distribution of a code allows the computation of the probability of error detection and correction with respect to some algorithms \cite{Klov}. Moreover, the weight distribution is always related to interesting and challenging problems in number theory (\cite{cal,Schroof}).

For irreducible cyclic codes, an identity due to McEliece \cite{McE74} shows that the weights of the codes can be expressed via Gauss sums. Because Gauss sums in general are extremely difficult to evaluate, the weight distribution of irreducible cyclic codes is still quite difficult to obtain, however, extensive studies have been carried out with much success by various number theoretic techniques (\cite{AL06,BM72,BM73,fit,McE74,McE72,Rao10,van,wol}). In particular nice characterizations were given in \cite{D-Y12,Vega1,Vega2} for irreducible cyclic codes with exactly one nonzero weight; necessary and sufficient conditions were provided and conjectures were also raised by Schmidt and White \cite{schmidt} for irreducible cyclic codes with at most two nonzero weights. Interested readers may consult the survey paper \cite{D-Y12} for more updated information on the weight distribution of irreducible cyclic codes.

For reducible cyclic codes, it has been known that the determination of weight distribution involves the evaluation of exponential sums. This may be even more difficult in general. For many special families of reducible cyclic codes where neat expressions are available, various delicate techniques from number theory and algebraic combinatorics have been developed and utilized, and for some of such families, the weight distribution can been obtained (see for example \cite{Ding2,FL08,Feng12,F-M12,holl,luo2,luo3,luo4,Ding1,M09,Mois09,Vega12,Tang12,Xiong1,Xiong2,Xiong3,zeng}). However, to our best knowledge, most of these literature works focus on reducible cyclic codes with two or three zeros. The exponential sums which have been explicitly evaluated seem to share a common feature that they attain only a few distinct values. %appearing in those case share a common feature that they are structured in some simple ways so that they can be explicitly evaluated by various techniques from number theory. %There are very few works on the weight distribution of reducible cyclic codes with more than three zeros.

For reducible cyclic codes with more than three zeros, not much is known. In a beautiful work \cite{gegeng}, the authors obtained the weight distribution of a class of cyclic codes with arbitrary number of zeros. Their work built upon an unexpected connection between the corresponding exponential sums and the spectra of Hermitian forms graphs which were known in the literature. In another recent work \cite{Y-X-D12} a general family of reducible cyclic codes with arbitrary number of zeros were constructed and under certain conditions the weight distribution was also obtained. The purpose of this paper is to explore the construction of \cite{Y-X-D12} much further and to determine the weight distribution for another new family of reducible cyclic codes with arbitrary number of zeros. Compared with \cite{Y-X-D12}, we achieve our goal by more advanced theory of Jacobi sums and by more subtle treatment of some complicated combinatorial identities.

The rest of the paper is organized as follows. The codes we consider will be introduced in Section \ref{sec-II}, so are the main results (Theorems 1, 2 and 3). Section \ref{sec-pre} introduces some mathematical tools such as cyclotomy, Gaussian periods and general Jacobi sums that will be needed later. In Sections \ref{sec-main} and \ref{sec-mainII} we prove our main theorems. To streamline the proofs of Theorems \ref{thm-e1}, \ref{thm-e2} and \ref{thm-e3} we have left out the proof of a complicated combinatorial identity to Section \ref{sec-app}. Section \ref{sec-conclusion} concludes this paper.

\section{Weight Distribution of Code $\mathcal{C}_{(a_1,\cdots,a_t)}$} \label{sec-II}

We first fix some notation. Let $p$ be a prime, $q=p^s$, $r=q^m$ for some integers $s,m\geqslant 1$. Let $\mathbb{F}_r$ be a finite field of order $r$ and $\gamma$ be a generator of the multiplicative group $\mathbb{F}_{r}^*:=\mathbb{F}_r \setminus \{0\}$. For any $t \ge 2$, the family of reducible cyclic codes $\mathcal{C}_{(a_1,\cdots,a_t)}$ with $t$ zeros were introduced in \cite{Y-X-D12} as follows.

\noindent \emph{For any $e \geqslant t \geqslant 2$, assume that}
\begin{itemize}
\item[ i)]  \emph{$a \not \equiv 0 \pmod{r-1} \mbox{ and } e|(r-1)$;}

\item[ ii)] \emph{$a_i \equiv a+\frac{r-1}{e}\Delta_i \pmod{r-1},\, 1\leqslant i \leqslant t$, where $\Delta_i \not \equiv \Delta_j \pmod{e}$ for any $ i \ne j$ and \\ $\gcd(\Delta_2-\Delta_1,\ldots,\Delta_t-\Delta_1,e)=1$;}

\item[ iii)] \emph{$
\deg h_{a_1}(x)=\cdots=\deg h_{a_t}(x)=m, \mbox{ and } h_{a_i}(x) \neq h_{a_j}(x)$ for any $1\leqslant i\neq j\leqslant t$, where $h_{a_i}(x)$ is the minimal polynomial of $\gamma^{-a_i}$ over $\mathbb{F}_q$};

\item[iv)] \emph{$N=\gcd \left(\frac{r-1}{q-1},a e\right)$};

\item[v)] \emph{$\delta=\gcd(r-1,a_1,a_2,\cdots ,a_{t}),\ n=\frac{r-1}{\delta}$}.
\end{itemize}
\emph{The cyclic code $\mathcal{C}_{(a_1,\cdots,a_t)}$ with $t$ zeros $\gamma^{-a_1},\cdots,\gamma^{-a_t}$ is given by}
\begin{equation}\label{def}
\begin{array}{l}
\mathcal{C}_{(a_1,\cdots,a_t)}=\left\{ c(x_1,x_2,\cdots,x_{t})=\left(\tr_{r/q}\left(\sum_{j=1}^t x_j \gamma^{a_ji}  \right)\right)_{i=0}^{n-1}~:~\forall \, x_1,\cdots,x_{t}\in\mathbb{F}_{r} \right\},
\end{array}\end{equation}
\emph{where $\tr_{r/q}$ denotes the trace map from $\mathbb{F}_{r}$ to $\mathbb{F}_{q}$.}

It shall be noted that Condition iii) can be easily verified, for example, it holds if $\frac{r-1}{q^\ell - 1}\nmid N$ for any proper factor $\ell$ of $m$ (i.e. $\ell \mid m$ and $\ell<m$, see \cite[Lemma 6]{Y-X-D12}). In particular this is always the case if $N=2$, which is our interest in the paper.

Delsarte's Theorem \cite{Delsarte} states that $\mathcal{C}_{(a_1,\cdots,a_t)}$ is an $[n,tm]$ cyclic code over $\mathbb{F}_{q}$ with parity-check
polynomial $h(x)=h_{a_1}(x)\cdots h_{a_{t}}(x)$. This class of codes $\mathcal{C}_{(a_1,\cdots,a_t)}$ contain many interesting cyclic codes as special cases which have been extensively studied in the literature (\cite{Ding1,Ding2,F-M12,Tang12,Xiong1,Xiong2,Xiong3}), all of which focus on the case $t=2$. %, and the weight distribution is known only for a few cases, which all require that $e \le 4$ (except for $N=1$). The problem of finding the weight distribution of $\mathcal{C}_{(a_1,\cdots,a_t)}$ for $t \ge 3$ is more difficult.

For any $t \ge 3$, in \cite{Y-X-D12} we obtain the weight distribution of $\mathcal{C}_{(a_1,\cdots,a_t)}$ under either of the following conditions:

\begin{itemize}
\item for any $t,e \ge 2$ when $N=1$; or

\item for any $t=e \ge 2$ with $N=1,2,3$; or with $N=(p^j+1)/k$ for some positive integers $j,k$; or with $N$ being a prime number such that $N \equiv 3\pmod{4},\lf{p}{N}=1$ (here $\lf{*}{*}$ denotes the Legendre symbol).     %\item  $e=3,t=2$ and $N=3$ \cite{Xiong2}.
%\item  $e=3,t=2$ and $N=(p^j+1)k$ for some positive integers $j,k$ \cite{Xiong3}.
%\item  $e=4,t=2$ and $N=2$ \cite{Xiong1}
\end{itemize}
In this paper we obtain the weight distribution of $\mathcal{C}_{(a_1,\cdots,a_t)}$ for any $t \ge 2$ such that $t=e-1$ and $N=2$. Note that under these conditions, it is necessary that $q$ is odd, $m$ is even and $2|ae$. Our main results are stated as follows.

\begin{thm}\label{thm-e1}
For $N=2$ and any $t=e-1 \geqslant 2$, we further assume that
\begin{eqnarray} \label{1:assumption}
e|(q^{m/2}-1) \, \mbox{ \emph{and} } \, 2|a.
\end{eqnarray}
Then $\mathcal{C}_{(a_1,\ldots,a_t)}$ is an $[n,tm]$ cyclic code over $\gf_q$ with the minimal Hamming distance $d=\frac{2(q-1)(r-\sqrt{r})}{(t+1)q\delta}$. It has (at most) $\frac{1}{2}(t^2+5t-2)$ nonzero distinct weights. \begin{itemize}
\item[(i).] If $q \equiv 1 \pmod{4}$, then the weight distribution is listed in Table \ref{Table1}.

\item[(ii).] If $q \equiv 3 \pmod{4}$, then the weight distribution is listed in Table \ref{Table2}.

\end{itemize}
\end{thm}

\begin{table}[ht]
\caption{The weight distribution of $\mathcal{C}$ when $N=2$ and $t=e-1\geqslant 2$: Case (i).}\label{Table1}
\begin{center}{
\begin{tabular}{|c|c|}
  \hline
  Weight & Frequency $\quad (\forall \, 1 \leqslant k \leqslant t, 0 \leqslant u\leqslant k+1)$ \\
  \hline
  \hline
$0$ & once \\
  \hline
  $\frac{q-1}{(t+1) q \delta}\cdot \bigg\{(k+1)r-(k+1-2u)\sqrt{r}\bigg\}$ & $\frac{(r-1)}{r 2^{k+2}}\cdot \binom{t+1}{k+1}\binom{k+1}{u}\cdot \bigg\{2(r-1)^k-$\\
  & $(-1)^k\left\{
  (1+\sqrt{r})^u(1-\sqrt{r})^{k+1-u}+(1-\sqrt{r})^u(1+\sqrt{r})^{k+1-u}\right\}\bigg\}$\\
  \hline
\end{tabular}}
\end{center}
\end{table}

\begin{table}[ht]
\caption{The weight distribution of $\mathcal{C}$ when $N=2$ and $t=e-1\geqslant 2$: Case (ii).}\label{Table2}
\begin{center}{
\begin{tabular}{|c|c|}
  \hline
  Weight & Frequency $\quad (\forall \, 1 \leqslant k \leqslant t, 0 \leqslant u\leqslant k+1)$ \\
  \hline
  \hline
$0$ & once \\
  \hline
  $\frac{q-1}{(t+1) q \delta}\cdot \bigg\{(k+1)r-(-1)^{m/2}(k+1-2u)\sqrt{r}\bigg\}$ & $\frac{(r-1)}{r 2^{k+2}}\cdot \binom{t+1}{k+1}\binom{k+1}{u}\cdot \bigg\{2(r-1)^k-$\\
  & $(-1)^k\left\{
  (1+\sqrt{r})^u(1-\sqrt{r})^{k+1-u}+(1-\sqrt{r})^u(1+\sqrt{r})^{k+1-u}\right\}\bigg\}$\\
  \hline
\end{tabular}}
\end{center}
\end{table}

We remark that if $N=2$ and $t=e-1$ is even, then the condition (\ref{1:assumption}) will always be satisfied, so this settles the case completely. In particular the special case $N=2,e=3,t=2$ was already studied in \cite{Tang12}. When $N=2$ and $t=e-1$ is odd, there are two cases: if the condition (\ref{1:assumption}) is satisfied, this is again settled by Theorem \ref{thm-e1}; on the other hand, if the condition (\ref{1:assumption}) is not satisfied, in principle the weight distribution can still be obtained. However, the formulas become quite complicated. To illustrate that, we first present the weight distribution for the simple case $t=3$ in Theorem \ref{thm-e2}, and then give a computational formula for the general case in Theorem \ref{thm-e3}.

\begin{thm}\label{thm-e2}
For $N=2$ and $t=e-1 =3$.
Then $\mathcal{C}_{(a_1,a_2,a_3)}$ is an $[n,3m]$ cyclic code over $\gf_q$ with the minimal Hamming distance $d=\frac{(q-1)(r-\sqrt{r})}{2q\delta}$. It has (at most) $12$ nonzero weights. \begin{itemize}
\item[(i).] If $2|a$, then its weight distribution is listed in Table \ref{Table3} (or Table \ref{Table1} with $t=3$).

\item[(ii).] If $2 \nmid a$, then its weight distribution is listed in Table \ref{Table4}.

\end{itemize}
\end{thm}

\begin{table}[ht]
\caption{The weight distribution of $\mathcal{C}_{(a_1,a_2,a_3)}$ when $t=e-1=3,N=2$ and $2\mid a$.}\label{Table3}
\begin{center}{
\begin{tabular}{|c|c|}
  \hline
  Weight & Frequency \\\hline\hline
  0& once\\\hline
 $\frac{q-1}{2\delta q}(r+\sqrt{r})$&$3(r-1)$~times\\\hline
 $\frac{q-1}{2\delta q}(r-\sqrt{r})$&$3(r-1)$~times\\\hline
 $\frac{3(q-1)}{4\delta q}(r+\sqrt{r})$ & $(r-1)(r-5)/2$~times \\\hline
 $\frac{3(q-1)}{4\delta q}(r-\sqrt{r})$ & $(r-1)(r-5)/2$~times \\\hline
 $\frac{(q-1)}{4\delta q}(3r+\sqrt{r})$ & $3(r-1)^2/2$~times \\\hline
 $\frac{(q-1)}{4\delta q}(3r-\sqrt{r})$ & $3(r-1)^2/2$~times \\\hline
 $\frac{(q-1)}{\delta q}(r+\sqrt{r})$ & $(r-1)(r^2-2r+9)/16$~times \\\hline
 $\frac{(q-1)}{\delta q}(r-\sqrt{r})$ & $(r-1)(r^2-2r+9)/16$~times \\\hline
 $\frac{(q-1)}{2\delta q}(2r+\sqrt{r})$ & $(r-1)(r^2-4r+3)/4$~times \\\hline
 $\frac{(q-1)}{2\delta q}(2r-\sqrt{r})$ & $(r-1)(r^2-4r+3)/4$~times \\\hline
 $\frac{(q-1)}{\delta q}r$ & $3(r-1)^3/8$~times \\\hline
\end{tabular}}
\end{center}
\end{table}

\begin{table}[ht]
\caption{The weight distribution of $\mathcal{C}_{(a_1,a_2,a_3)}$ when $t=e-1=3,N=2$ and $2\nmid a$.}\label{Table4}
\begin{center}{
\begin{tabular}{|c|c|}
  \hline
  Weight & Frequency \\\hline\hline
  0& once\\\hline
 $\frac{q-1}{2\delta q}(r+\sqrt{r})$&$(r-1)$~times\\\hline
 $\frac{q-1}{2\delta q}(r-\sqrt{r})$&$(r-1)$~times\\\hline
 $\frac{q-1}{2\delta q}r$&$4(r-1)$~times\\\hline
 $\frac{3(q-1)}{4\delta q}(r+\sqrt{r})$ & $(r-1)^2/2$~times \\\hline
 $\frac{3(q-1)}{4\delta q}(r-\sqrt{r})$ & $(r-1)^2/2$~times \\\hline
 $\frac{(q-1)}{4\delta q}(3r+\sqrt{r})$ & $(r-1)(3r-7)/2$~times \\\hline
 $\frac{(q-1)}{4\delta q}(3r-\sqrt{r})$ & $(r-1)(3r-7)/2$~times \\\hline
 $\frac{(q-1)}{\delta q}(r+\sqrt{r})$ & $(r-1)^3/16$~times \\\hline
 $\frac{(q-1)}{\delta q}(r-\sqrt{r})$ & $(r-1)^3/16$~times \\\hline
 $\frac{(q-1)}{2\delta q}(2r+\sqrt{r})$ & $(r-1)(r^2-4r+3)/4$~times \\\hline
 $\frac{(q-1)}{2\delta q}(2r-\sqrt{r})$ & $(r-1)(r^2-4r+3)/4$~times \\\hline
 $\frac{(q-1)}{\delta q}r$ & $(r-1)(3r^2-6r+11)/8$~times \\\hline
\end{tabular}}
\end{center}
\end{table}

We now consider the general case for $N=2$ and $t=e-1$. Denote $g:=\gamma^a$ and $\beta:=\gamma^{(r-1)/e}$. And let $A$ be the Vandermonde matrix of size $(t+1)\times (t+1)$, given by
\begin{equation}\label{matrix}A:=\left(
         \begin{array}{llll}
           1&1&\cdots&1 \\
           1&\beta&\cdots&\beta^{e-1} \\
           1&\beta^2&\cdots&\beta^{2(e-1)} \\
           \vdots &\vdots &&\vdots\\
           1&\beta^{e-1}&\cdots&\beta^{(e-1)^2}
         \end{array}
       \right),\end{equation}
Take $B$ be the $(t+1) \times t$-matrix whose columns consist of the $\{\triangle_1+1,\ldots,\triangle_t+1\} \pmod{e}$ columns of $A$, where $\triangle_i$ are the basic parameters of $\mathcal{C}_{(a_1,\cdots,a_t)}$. Let
\begin{equation}\label{equ-xtoy}
    (y_0,\cdots ,y_{t})^T=B(x_1,\cdots ,x_t)^T.
\end{equation}
Since rank$B=t$ (see also \cite[Lemma 18]{Y-X-D12}), this gives a one-to-one correspondence between $(y_1,\ldots,y_t)$ and $(x_1,\ldots,x_t)$, and there exist some $0 \ne \lambda_i \in \mathbb{F}_{q}$ for $1\leqslant i\leqslant t$ such that
\begin{equation*}\label{equ-y0}
  y_0+\sum\limits_{i=1}^t \lambda_iy_i=0
\end{equation*}
We note that $\{\lambda_i\}_{i=1}^t$ depend only on the parameters $\{\triangle_i \pmod{e}\}_{i=1}^t$ and $\beta$. We further define
\begin{equation}\label{equ-l1l2}
  \begin{array}{l}
  l_0=\#\{i\mid\lambda_ig^i \mbox{ is a square in } \mathbb{F}_r,1\leqslant i\leqslant t\};\\
  l_1=\#\{i\mid\lambda_ig^i \mbox{ is a nonsquare in } \mathbb{F}_r,1\leqslant i\leqslant t\}.
  \end{array}
\end{equation}
Next, we extend the definition of binomial coefficient to all integers such that
$$\binom{n}{i}=0,\ \mbox{for $i<0$ and $i>n$}.$$

With such preparations, we give our main result for the general case as follows.

\begin{thm}\label{thm-e3}
In the case of $N=2$ and $t=e-1\geqslant 2$, the $\mathcal{C}_{(a_1,\ldots,a_t)}$ is an $[n,tm]$ cyclic code over $\gf_q$ with the minimal Hamming distance $d=\frac{2(q-1)(r-\sqrt{r})}{(t+1)q\delta}$, and the Hamming weight of its codewords takes the value 0 once and the value 
\begin{eqnarray*}  \frac{(q-1)}{(t+1)q \delta}\left[k(r-1)-2u\eta_0^{(2,r)}-2(k-u)\eta_1^{(2,r)}\right], \end{eqnarray*}
for any $2 \le k \le t+1$ and $0 \leqslant  u \le k$, with the frequency
\begin{eqnarray*} \sum_{k_0=0}^k\sum_{u_0=0}^u\binom{l_0+1}{k_0}\binom{l_1}{k-k_0}
\binom{k_0}{u_0}\binom{k-k_0}{u-u_0}\Omega_{{\scriptsize \underbrace{0 \cdots 0}_{2u_0+k-k_0-u,} \underbrace{1\cdots 1}_{k_0+u-2u_0}}}. \end{eqnarray*}
where $\eta_0^{(2,r)},\eta_1^{(2,r)}$ are given by Lemma \ref{lem-degree2}, $\Omega_{{\scriptsize \underbrace{0 \cdots 0}_{u} \underbrace{1\cdots 1}_{v}}}$ is determined by Lemma \ref{thm-NNN3} and $l_0,l_1$ are defined by (\ref{equ-l1l2}).
\end{thm}

We remark that Theorem \ref{thm-e3} is a general computational formula for the weight distribution of $\mathcal{C}_{(a_1,\cdots,a_t)}$, and
the results of Theorem \ref{thm-e1} and Theorem \ref{thm-e2} can be viewed as its corollaries. However, the frequency formula in Theorem \ref{thm-e3} is complicated since it depends on the choice of $\triangle_1,\cdots,\triangle_t$, and there seems no easy way to write them down in a simple closed form as Theorem \ref{thm-e1}.

In the end of this section, we give serval numerical examples to illustrate the our main theorems.

\begin{exam}
Let $(q, m, e, t)=(5,2,4,3)$. Then $\frac{r-1}{3}=\frac{5^2-1}{4}=6$. Let $\gamma$ be the generator of $\gf_{25}^*$ with characteristic polynomial $\gamma^2 + 4\gamma + 2=0$. Let $(\triangle_1,\triangle_2,\triangle_3)=(1,2,3)$.
\begin{itemize}
\item[(1).] For $a=2$ we have $(a_1,a_2,a_3)=(8,14,20)$, $(\delta,n)=(2,12)$ and
$$
h_{a_1}(x)=x^2 + x + 1, \, h_{a_2}(x)= x^2 + 3x + 4, \, h_{a_3}(x)= x^2 + 4x + 1.
$$
The parity-check polynomial of $\mathcal{C}$ is then
$
h(x)=x^6 + 3x^5 + 3 x^3+3x+4.
$
The code $\mathcal{C}$ is a $[12,6,4]$-cyclic code over $\gf_5$ with weight enumerator given by
\[1+72Y^4+312 Y^{6}+864 Y^{7}+1740 Y^{8}+3408 Y^{9}+5184 Y^{10}+3168 z^{11}+876 Y^{12}.\]
This also follows from Table III. There are 8 distinct non-zero weights because some of the weights in Table III turn out the same. More precisely,
\begin{eqnarray} \label{1:eqns} \begin{array}{ccc}
\frac{(q-1)}{2 \delta q}(r+\sqrt{r}) =\frac{3(q-1)}{4 \delta q}(r-\sqrt{r}), \\
\frac{(q-1)}{4 \delta q}(3r+\sqrt{r})=\frac{(q-1)}{\delta q}(r-\sqrt{r}),\\
\frac{3(q-1)}{4 \delta q}(r+\sqrt{r})=\frac{(q-1)}{2 \delta q}(2r-\sqrt{r}).\end{array}
\end{eqnarray}

\item[(2).] For $a=1$ we have $(a_1,a_2,a_3)=(7,13,19)$, $(\delta,n)=(1,24)$ and
$$
h_{a_1}(x)=x^2 + x + 2, \, h_{a_2}(x)= x^2 + 2x + 1, \, h_{a_3}(x)= x^2 + 4x + 2.
$$
The parity-check polynomial of $\mathcal{C}$ is then
$
h(x)=x^6 + 2x^5 + 4x^4 + x^3+2x^2+3x+4.
$
The code $\mathcal{C}$ is a $[24,6,8]$-cyclic code over $\gf_5$ with weight enumerator given by
\[1+24Y^8+96 Y^{10}+312 Y^{12}+816 Y^{14}+1680 Y^{16}+3456 Y^{18}+5208 z^{20}+3168 Y^{22}+ 864 Y^{24}.\]
This also follows from Table IV. There are 9 distinct non-zero weights because some of the weights in Table IV turn out the same. More precisely, the equations (\ref{1:eqns}) still hold true.

\end{itemize}
\end{exam}

\section{Cyclotomy, Gaussian periods and Jacobi sums}\label{sec-pre}

An {\em additive character} of $\mathbb{F}_{r}$ is a nonzero function $\phi$
from $\mathbb{F}_{r}$ to the set of complex numbers such that
$\phi(x+y)=\phi(x) \phi(y)$ for any pair $(x, y) \in \mathbb{F}_{r}^2$. Let $\tr_{r/p}$ denote the trace function from $\mathbb{F}_{r}$ to $\mathbb{F}_{p}$ and $\zeta_p=e^{2\pi \sqrt{-1}/p}$ be the primitive $p$-th complex root of unit. The additive character $\psi$ given by
\begin{eqnarray}\label{dfn-add}
 \psi(c)=\zeta_p^{\tr_{r/p}(c)} \ \ \mbox{ for any }
c\in\mathbb{F}_{r}
\end{eqnarray}
is called the {\em canonical additive character} of $\mathbb{F}_{r}$. For any $x\in \mathbb{F}_{r}$, one can easily check the orthogonal property
\begin{equation}\label{add-orth}
    \frac{1}{r}\sum\limits_{x\in \mathbb{F}_{r}}\psi(ax)=\left\{
      \begin{array}{ll}
        1, & \hbox{if $a=0$;} \\
        0, & \hbox{if $a\in \mathbb{F}^*_{r}$.}
      \end{array}
    \right.
\end{equation}

Let $r-1=l L$ for two positive integers $l, L\geqslant 1$, and let
$\gamma$ be a fixed primitive element of $\mathbb{F}_{r}$.
Define $C_{i}^{(L,r)}=\gamma^i \langle \gamma^{L} \rangle$ for $i=0,1,...,L-1$, where
$\langle \gamma^{L} \rangle$ denotes the
subgroup of $\mathbb{F}_{r}^*$ generated by $\gamma^{L}$. The $C_{i}^{(L,r)}$ are
called the {\em cyclotomic classes} of order $L$ in $\mathbb{F}_{r}$. The {\em Gaussian periods} of order $L$ are defined by
$$
\eta_i^{(L,r)} =\sum_{x \in C_i^{(L,r)}} \psi(x), \quad i=0,1,..., L-1.
$$

The values of Gaussian periods are difficult to compute in general.
However, they are known in a few cases. We will need the following whose proofs can be found in \cite{B-E-W} and \cite{Myer}.

\vspace{.2cm}
\begin{lemma}\label{lem-degree2}
When $L=2$, the Gaussian periods are given by
\begin{eqnarray*}
\eta_0^{(2,r)}=
\left\{
\begin{array}{ll}
\frac{-1+(-1)^{s\cdot m-1}r^{1/2}}{2}, & \mbox{if $p\equiv 1 \pmod{4}$} \\
\frac{-1+(-1)^{s\cdot m-1}(\sqrt{-1})^{s\cdot m} r^{1/2}}{2}, & \mbox{if $p\equiv 3 \pmod{4}$}
\end{array}
\right.
\end{eqnarray*}
and
$\eta_1^{(2,r)} = -1 - \eta_0^{(2,r)}.$
\end{lemma}

A {\em multiplicative character} of $\gf_r$ is a nonzero function
$\chi$ from $\gf_r^*$ to the set of complex numbers such that
$\chi(xy)=\chi(x)\chi(y)$ for all the pairs $(x, y) \in \gf_r^*
\times \gf_r^*$. For $j=1,2,\ldots,r-1$, one can easily check that the functions $\chi^{(j)}$ with
\begin{eqnarray}\label{dfn-mul}
\chi^{(j)}(\gamma^k)=\zeta_{r-1}^{jk} \ \ \mbox{for } k=0,1,\ldots,r-2 \nonumber
\end{eqnarray}
give all the multiplicative character of order dividing $r-1$, here $\zeta_{r-1}$ denotes the primitive complex $(r-1)$-th root of unit. When $j=r-1$, $\varepsilon(c):=\chi^{(r-1)}(c)=1 \mbox{ for all }
c\in\gf_r^*,$ which is called the {\em trivial multiplicative
character} of $\gf_r$. One can check the following orthogonal property of multiplicative characters
\begin{equation}\label{equ-orthogonal1}
    \frac{1}{r-1}\sum\limits_{x\in \mathbb{F}_{r}^*}\chi(x)=
\left\{
  \begin{array}{ll}
    1, & \hbox{if\ }\chi=\varepsilon; \\
    0, & \hbox{otherwise.}
  \end{array}
\right.
\end{equation}
Furthermore, we may extend the definition of any multiplicative character $\chi$ to $\mathbb{F}_{r}$ as follows,
\begin{equation}\label{equ-chi(0)}
    \chi(0)=\left\{
              \begin{array}{ll}
                0, & \hbox{if }\chi\neq \varepsilon; \\
                1, & \hbox{if }\chi=\varepsilon.
              \end{array}
            \right.\nonumber
\end{equation}

%Let $\chi$ be a multiplicative character with order $L$ where $L|(q-1)$. Then the {\em Gaussian sum} $G(\chi)$ %of order $L$ is defined by
%\begin{eqnarray}\label{def-G}
%G(\chi):=\sum_{c\in\gf_r^*} \chi(c)\zeta_p^{\tr(c)}. \nonumber
%\end{eqnarray}

Let $k\geqslant 2$ and $\chi_1,\cdots ,\chi_k$ be multiplicative characters of $\mathbb{F}_{r}$. The \textit{Jacobi sum} related with $\chi_1,\cdots ,\chi_k$ over $\mathbb{F}_{r}$ is defined by
$$J(\chi_1,\cdots ,\chi_k):=\sum\limits_{z_1,\cdots z_k\in \mathbb{F}_{r}\atop z_1+\cdots +z_k=1}\chi_1(z_1) \chi_2(z_2)\cdots \chi_k(z_k).$$

The following (\cite{B-E-W}) are elementary properties of Jacobi sums.
\begin{lemma}\label{lem_J}
\par (a).~$J(\underbrace{\varepsilon,\cdots ,\varepsilon}_k)=q^{k-1}$.
\par (b).~$J(\chi_1,\cdots,\chi_k)=0$ if some but not all of $\chi_1,\cdots ,\chi_k$ are trivial.
\par (c).~When $r$ is odd, let $\rho$ be the quadratic multiplicative character of $\mathbb{F}_{r}$, then
 $$J(\underbrace{\rho,\cdots,\rho}_k)=
\left\{
  \begin{array}{ll}
    -\rho(-1)^{\frac{k}{2}}r^{\frac{k-2}{2}}, & \hbox{if $k$ is even;} \\
    \rho(-1)^{\frac{k-1}{2}}r^{\frac{k-1}{2}}, & \hbox{if $k$ is odd.}
  \end{array}
\right.$$
 \end{lemma}

We now define the \textit{reduced Jacobi sums} below, which is needed in the next section.
\begin{equation}\label{equ-J*}
    J^*(\chi_1,\cdots ,\chi_k):=\sum\limits_{z_1,\cdots z_k\in \mathbb{F}^*_{r}\atop z_1+\cdots +z_k=1}\chi_1(z_1) \chi_2(z_2)\cdots \chi_k(z_k).
\end{equation}
Notice that $J^*(\chi_1,\cdots ,\chi_k)=J(\chi_1,\cdots ,\chi_k)$ if all of $\chi_1,\cdots ,\chi_k$ are non-trivial. The following results give the evaluation of $J^*(\chi_1,\cdots ,\chi_k)$ if some of $\chi_1,\cdots ,\chi_k$ are trivial. The next result is not difficult but may be of independent interest. It is essential in Section \ref{sec-app} to establish a complicated combinatorial identity, which is needed in the proofs of Theorems \ref{thm-e1} and \ref{thm-e2}.
\begin{lemma}\label{lem-J*}
(a).~$J^*(\varepsilon,\cdots ,\varepsilon)=\left\{(r-1)^k-(-1)^k\right\}/r$.
%\sum\limits_{i=0}^{k-1}(-1)^i \binom{k}{i}r^{k-1-i}$.
\par(b).~Define $J(\chi):=1$ for any multiplicative character $\chi$. Let $u$ be an integer such that $0\leqslant u\leqslant k-1$. If $\chi_1,\cdots ,\chi_{k-u}$ are all nontrivial multiplicative characters, then
$$J^*(\chi_1,\cdots ,\chi_{k-u},\underbrace{\varepsilon,\cdots ,\varepsilon}_{u})=(-1)^uJ(\chi_1,\cdots ,\chi_{k-u}).$$
\end{lemma}
\begin{proof} By definition, we have
$$J^*(\varepsilon,\cdots ,\varepsilon)=J(\varepsilon,\cdots ,\varepsilon)-\sum\limits_{\mathcal{I}} \sum\limits_{\sum\limits_{i\in \mathcal{I}} z_i=1}\varepsilon(\prod\limits_{i\in \mathcal{I}} z_i),$$
where the subscript $\mathcal{I}$ under the $\sum$ symbol means to sum over all subsets $\mathcal{I}$ such that $\mathcal{I}\subsetneqq\{1,2,\cdots ,k\}$. Using the Inclusion-exclusion principle, Part (a) of Lemma \label{lem-J*} can be easily proved. Now for Part (b), we have $\mathcal{I}'\subsetneqq\{k-u+1,\cdots ,k\}$, then
$$\begin{array}{l}J^*(\chi_1,\cdots ,\chi_{k-u},\underbrace{\varepsilon,\cdots ,\varepsilon}_{u})\\
\ =J(\chi_1,\cdots ,\chi_{k-u},\underbrace{\varepsilon,\cdots ,\varepsilon}_{u})-\sum\limits_{\mathcal{I}'}\sum\limits_{\sum_{j=1}^{k-u} z_j+\sum\limits_{i\in\mathcal{I}'}z_i=1}\chi_1(z_1)\chi_2(z_2)\cdots\chi_k(z_{k-u})\varepsilon(\prod\limits_{i\in \mathcal{I}'}z_i)\\
\ =0-\binom{u}{1}J^*(\chi_1,\cdots ,\chi_{k-u},\underbrace{\varepsilon,\cdots ,\varepsilon}_{u-1})- \binom{u}{2}J^*(\chi_1,\cdots ,\chi_{k-u},\underbrace{\varepsilon,\cdots ,\varepsilon}_{u-2})-\cdots -
\binom{u}{u}J^*(\chi_1,\cdots ,\chi_{k-u}). \end{array}$$
By induction, Part (b) can be also verified.
\end{proof}

\section{Proof of Theorem \ref{thm-e1}}\label{sec-main}

\subsection{The weight distribution of $\mathcal{C}_{(a_1,\cdots ,a_t)}$ and Summation of Gaussian periods}\label{sec-Weight of C}

We now consider the weight distribution of the cyclic code $\mathcal{C}_{(a_1,\cdots,a_t)}$ given in (\ref{def}).
Using the orthogonal relation (\ref{add-orth}) and some computational techniques, in \cite{Y-X-D12} we haved expressed the Hamming weight of the codeword $c(x_1,\cdots,x_{t})$ by
\begin{eqnarray}\label{equ-Weight of C}
w_H(c(x_1,\cdots,x_{t})) =\frac{(r-1)(q-1)}{q \delta}-\frac{N(q-1)}{ eq \delta}\sum\limits_{h=0}^{e-1}
\bar\eta^{(N,r)}_{g^{h}\cdot\sum\limits_{\tau=1}^{t} x_\tau \beta_\tau^{h}},
\end{eqnarray}
where $g=\gamma^a$, $\beta_\tau=\gamma^{\frac{r-1}{e}\Delta_\tau}$ for $1\leqslant \tau\leqslant t$ and $\bar\eta^{(N,r)}_{v}=\sum\limits_{z\in C_{0}^{(N,r)}}\psi(vz)$ for any $v\in\mathbb{F}_{r}$. These $\bar\eta^{(N,r)}_v$ are called the \textit{modified Gaussian periods}, given by
$$\left\{
    \begin{array}{l}
     \bar\eta^{(N,r)}_0=\frac{r-1}{N}\\
     \bar\eta^{(N,r)}_{\gamma^{i}}=\eta_i^{(N,r)}\quad \hbox{ for $0\leqslant i\leqslant N-1$,}
    \end{array}
  \right.$$
where these $\eta_i^{(N,r)}$ are the ordinary Gaussian periods. Thus, to compute the weight distribution of cyclic code $\mathcal{C}_{(a_1,\cdots,a_t)}$, it suffices to compute the value distribution of the sum
\begin{equation}\label{equ-Tx}
T(x_1,\cdots ,x_{t}):=\sum\limits_{h=0}^{e-1}\bar\eta^{(N,r)}_{g^{h}
\cdot\sum_{\tau=1}^{t} x_\tau \beta_\tau^{h}},\quad (\forall x_1,\cdots ,x_t\in \mathbb{F}_{r}).\end{equation}
Now we deal with it under the assumption of $N=2$ and $t=e-1 \ge 2$.

\subsection{$N=2$ and $t=e-1 \ge 2$}
Since $N=2$, it is easy to see that $q$ is odd, $m$ is even and $-1=\gamma^{\frac{q^m-1}{2}}$ is a square. For simplicity, let us write
\[\bar\eta_x:=\bar\eta^{(2,r)}_x, \quad \forall x \in \gf_r. \]
Make a change of variables
\[y_h=\sum_{\tau=1}^{t} x_\tau \beta_\tau^{h}, \quad 0 \le h \le t=e-1, \]
which can be written as
\begin{equation}\label{equ-xtoy}
    (y_0,\cdots ,y_{t})^T=B(x_1,\cdots ,x_t)^T
\end{equation}
for some $(t+1) \times t$ matrix $B$. Recall that $\beta=\gamma^{(r-1)/e}$ is an $e$-th root of unity in $\mathbb{F}_{r}$.
Since $\beta_{\tau}=\beta^{\triangle_{\tau}}$, the matrix $B$ consists of $t$ columns of the Vandermonde matrix $A$, defined by (\ref{matrix}).
By \cite[Lemma 18]{Y-X-D12}, any $t$ rows of $B$ are linearly independent over $\gf_q$. This gives a one-to-one correspondence between $(y_1,\ldots,y_t)$ and $(x_1,\ldots,x_t)$ and a relation
\begin{eqnarray} \label{3:xiong} y_0+ \sum_{h=1}^t\lambda_hy_h=0, \quad  \mbox{ for some } 0 \ne \lambda_h \in \gf_{q^{m}} \, \forall h. \end{eqnarray}
We define $\tilde{\lambda}_h\ (1\leqslant h\leqslant t)$ as
\begin{equation}\label{equ-lambda}
  \tilde{\lambda}_h=\left\{\begin{array}{ll}1,&\mbox{if $\lambda_hg^h$ is a square in $\mathbb{F}_{r}$;}\\
  \gamma,&\mbox{if $\lambda_hg^h$ is a nonsquare in $\mathbb{F}_{r}$,}\end{array}\right.
\end{equation}
and we change variables again $\lambda_hy_h \to y_h$, then we see that to compute the weight distribution of the cyclic code $\mathcal{C}_{(a_1,\cdots,a_t)}$, it suffices to compute the value distribution of the sum
\begin{equation}\label{equ-Ty}
\widetilde{T}(y_0,\cdots ,y_{t}):=\bar\eta_{y_0}+\sum\limits_{h=1}^{t}\bar\eta_{\tilde{\lambda}_hy_h},\quad \forall \, (y_1,\cdots ,y_t)\in \mathbb{F}_{r}^t,\end{equation}
where $y_0:=y_0(y_1,\ldots,y_t)$ satisfies 
\begin{equation}\label{eqn-y0}
  y_0+\sum\limits_{h=1}^t y_h=0
\end{equation}

\subsection{Proof of Theorem \ref{thm-e1}.}\label{sec-sub_e4t3}

When $2|a$, then $g=\gamma^a$ is a square. Moreover, $e|(q^{m/2}-1)$ means that $\beta=\gamma^{(q^{m/2}+1)(q^{m/2}-1)/{e}} \in \gf_{q^{m/2}}$, hence the matrix $A$ is defined over $\gf_{q^{m/2}}$, so are all the $\lambda_h$ in (\ref{3:xiong}), thus $\lambda_h$ are all squares in $\gf_{q^m}$, that is, $\tilde{\lambda}_h=1\ (\forall h)$.

To study the value distribution of $\widetilde{T}:=\widetilde{T}(y_0,\ldots,y_t)$, we will divide the space of $(y_1,\ldots,y_t) \in \gf_r^t$ according to $s$, which counts the number of $i$'s ($0 \le i \le t$) such that $y_i=0$. Obviously $0 \le s \le t+1$.

If $s \ge t$, i.e., at least $t$ terms of $y_0,y_1,\ldots,y_t$ equal to 0, then all of them equal to 0, $\widetilde{T}=(t+1)\bar\eta_0$ and the frequency is 1.

If $s=t-1$, i.e., exactly $(t-1)$ terms of $y_0,y_1,\ldots,y_t$ equal to 0, say for example the two terms which are not 0 are $y_i,y_j$ for some $0 \le i < j \le t$, the number of choices of such $i,j$ is $\binom{t+1}{2}$, and the constraint (\ref{eqn-y0}) becomes $y_i+y_j=0$, or $y_i=-y_j$. Hence for this $i,j$ we find that
\[\widetilde{T}=(t-1)\bar\eta_0+\bar\eta_{y_j}+\bar\eta_{-y_j}=(t-1)\bar\eta_0+2\bar\eta_{y_j}.\]
So the value distribution of $\widetilde{T}$ for $s=t-1$ is as follows:
$$\left.
    \begin{array}{ll}
      \mbox{Value} \, \, \widetilde{T} & \hbox{Frequency} \\
      (t-1)\bar\eta_0+2\eta_0, & \frac{r-1}{2} \cdot \binom{t+1}{2} \\[2mm]
      (t-1)\bar\eta_0+2\eta_1, & \frac{r-1}{2} \cdot \binom{t+1}{2} \\    \end{array}
  \right.$$

Now suppose in general $s=t-k$ for some $k$ with $1 \le k \le t$. Say the $(k+1)$ terms which are not 0 are $y_{i_0},y_{i_1}, \ldots,y_{i_k}$ for some $0 \le i_0< i_1 < \cdots <i_k \le t$. The number of ways to choose such $i_j$'s is $\binom{t+1}{k+1}$, and for such $i_j$'s, the constraint (\ref{eqn-y0}) becomes
\[y_{i_0}+y_{i_1}+\cdots+y_{i_k}=0,\]
and we find that
\[\widetilde{T}=(t-k)\bar\eta_0+\bar\eta_{y_{i_0}}+\bar\eta_{y_{i_1}}+\cdots+\bar\eta_{y_{i_k}}.\]
In order to compute the value distribution of $\widetilde{T}$ for these cases, it suffices to compute for any positive integer $u$ and any sequence $i_1,\cdots,i_{u},i_{u+1}\in\{0,1\}$ the value $\Omega_{i_1\cdots i_ui_{u+1}}$ given by
\begin{equation}\label{equ-NNN}
 \Omega_{i_1\cdots i_ui_{u+1}}:=\#\left\{(x_1,\cdots ,x_u)\in (\mathbb{F}_{r}^*)^u\ \left|\ x_1\in C_{i_1}^{(2,r)},\cdots, x_u\in C_{i_u}^{(2,r)}, \sum\limits_{j=1}^ux_j\in C_{i_{u+1}}^{(2,r)}\right\}\right..
\end{equation}
We will prove in Section \ref{sec-app} that the value $\Omega_{i_1\cdots i_ui_{u+1}}$ depends only on the number of $0$'s and $1$'s in the sequence $i_1,\ldots,i_{u+1}$. More precisely for any $u+v \ge 1$ we have (see Lemma \ref{thm-NNN3} in Section \ref{sec-app})
\[\Omega_{{\scriptsize \underbrace{0 \cdots 0}_u \underbrace{1\cdots 1}_v}}=\frac{r-1}{r 2^{u+v+1}}\bigg\{2(r-1)^{u+v-1}+(-1)^{u+v} \left\{(1+\sqrt{r})^u(1-\sqrt{r})^v+(1-\sqrt{r})^u(1+\sqrt{r})^v\right\}\bigg\}.\]
Note that the number of ways to choose a fixed $u \ge 0$ is $\binom{k+1}{u}$. So, for the case that $s=t-k$, $1 \le k \le t$, the value distribution of $\widetilde{T}$ is given as follows
$$\left.
    \begin{array}{ll}
      \mbox{Value} \,\, \widetilde{T} & \hbox{Frequency} \, (\forall u,v \ge 0, u+v=k+1)\\
      (t-k)\bar\eta_0+u\eta_0+v \eta_1, & \binom{t+1}{k+1}\binom{k+1}{u} \Omega_{{\scriptsize \underbrace{0 \cdots 0}_u \underbrace{1\cdots 1}_v}} \\
    \end{array}
  \right.$$
As for the values $\bar\eta_0,\eta_0,\eta_1$, we have $\bar\eta_0=\frac{r-1}{2}$ and from Lemma \ref{lem-degree2}
\begin{eqnarray*}
\left\{
\begin{array}{lll}
\eta_0=\frac{-1-\sqrt{r}}{2}, & \eta_1=\frac{-1+\sqrt{r}}{2}, & \mbox{if $q\equiv 1 \pmod{4}$}, \\
\eta_0=\frac{-1-(-1)^{ms/2}\sqrt{r}}{2}, & \eta_1=\frac{-1+(-1)^{ms/2}\sqrt{r}}{2}, & \mbox{if $q\equiv 3 \pmod{4}$}. \\
\end{array}
\right.
\end{eqnarray*}
Now we have obtained the value distribution of $\widetilde{T}$. Returning to (\ref{equ-Tx}) and (\ref{equ-Weight of C}) gives us the weight distribution of the cyclic code $\mathcal{C}_{(a_1,\cdots,a_t)}$, which is summarized in Tables \ref{Table1} and \ref{Table2} in Theorem \ref{thm-e1}. This completes the proof of Theorem \ref{thm-e1}.

\section{Proof of Theorem \ref{thm-e2} and Theorem \ref{thm-e3}}\label{sec-mainII}

\subsection{Proof of Theorem \ref{thm-e3}}\label{sec-main3}

%\subsection{Study of the general case $N=2,t=e-1$ and $2 \nmid a$}
%We first study the weight distribution problem for the general case $t=e-1,N=2$ and $2 \nmid a$ and then derive Theorem \ref{thm-e2} as a special case $t=3$.
%Since $N=2$ and $2 \nmid a$, then $g=\gamma^a$ is not a square in $\gf_r$ and $t \ge 3$ is necessarily odd.

Recall from (\ref{equ-Ty}) and (\ref{eqn-y0}) that to compute the weight distribution of the cyclic code $\mathcal{C}_{(a_1,\cdots,a_t)}$, it suffices to compute the value distribution of the sum
\begin{equation}\label{equ-Ty2}
\widetilde{T}(y_0,y_1,\cdots,y_{l_0},\gamma z_1,\cdots ,\gamma z_{l_1}):=\sum\limits_{h=0}^{l_0}\bar\eta_{y_h}+
\sum\limits_{h=1}^{l_1}\bar\eta_{\gamma z_h},\quad \forall \, (y_1,\cdots,y_{l_0},z_1,\ldots,z_{l_1})\in \mathbb{F}_{r}^{l_0+l_1},\end{equation}
where $l_0,l_1$ are defined by (\ref{equ-l1l2}) so that $l_0+l_1=t$ and $y_0:=y_0(y_1,\ldots,y_{l_0},z_1,\ldots,z_{l_1})$ satisfies 
\begin{eqnarray} \label{eqn-y02}
y_0+y_1+\cdots+y_{l_0}+ z_1+\cdots+ z_{l_1}=0.
\end{eqnarray}
To study the value distribution of $\widetilde{T}$ in (\ref{equ-Ty2}), we consider the different subcases according to different $(k_0,k_1)$, where $k_0,k_1$ are defined by
$$\begin{array}{l}
  k_0:=\#\{i \mid 0\leqslant i\leqslant l_0, y_i\neq 0\}; \\
  k_1:=\#\{i \mid 1\leqslant i\leqslant l_1, z_i\neq 0\}.
\end{array}$$

If $k_0+k_1 \le 1$, by (\ref{eqn-y02}), all of $y_0,y_1,\ldots,y_{l_0},z_1,\ldots,z_{l_1}$ are 0, the frequency is 1 and $\widetilde{T}=(t+1) \bar\eta_0$.

If $k_0+k_1 \ge 2$, the number of ways to choose exactly $k_0$ non-zero terms in $y_0,\ldots,y_{l_0}$ and exactly $k_1$ non-zero terms in $z_1,\ldots,z_{l_1}$ is $\binom{l_0+1}{k_0}\binom{l_1}{k_1}$. Once they are chosen, without loss of generality we may assume that they are $y_1,\ldots,y_{k_0}$ and $z_1,\ldots,z_{k_1}$. Then in this case we have
\[\widetilde{T}=(t+1-k_0-k_1) \bar\eta_0+\sum_{i=1}^{k_0}\bar\eta_{y_i}+
\sum_{i=1}^{k_1}\bar\eta_{\gamma z_i},\]
and the constraint (\ref{eqn-y02}) becomes
\[y_1+\cdots+y_{k_0}+z_1+\cdots+z_{k_1}=0. \]

In order to compute the value distribution of $\widetilde{T}$ for these cases, let us consider for any $i_1,\ldots,i_{k_0},j_1,\ldots,j_{k_1} \in \{0,1\}$ the value $\Omega_{i_1\cdots i_{k_0};j_{1}\cdots j_{k_1}}^{'}$, given by
\begin{equation*}\label{equ-NNN2}
 \Omega_{i_1\cdots i_{k_0};j_{1}\cdots j_{k_1}}^{'}:=\#\left\{(y_1,\cdots ,y_{k_0};z_1,\cdots ,z_{k_1})\in (\mathbb{F}_{r}^*)^{k_0+k_1}\ \left|\ {y_{u}\in C_{i_{u}}^{(2,r)}, \gamma z_{v} \in C_{j_v}^{(2,r)}, 1 \le u \le k_0,1 \le v \le k_1}\atop {y_1+\cdots+y_{k_0}+z_1+\cdots+z_{k_1}=0}\right\}\right..
\end{equation*}
For any $ i \in \{0,1\}$, define $\bar{i} \in \{0,1\}$ by $\bar{i} \equiv i+1 \pmod{2}$. Clearly
\[\Omega_{i_1\cdots i_{k_0};j_{1}\cdots j_{k_1}}^{'}=\Omega_{i_1\cdots i_{k_0}\bar{j}_{1}\cdots \bar{j}_{k_1}}, \]
which is defined in (\ref{equ-NNN}) and is evaluated in Section \ref{sec-app}. In $\{i_1,\ldots,i_{k_0}\}$, let $u_0$ be the number of $0$'s and $u_1$ be the number of $1$'s; similarly, in $\{j_1,\ldots,j_{k_1}\}$, let $v_0$ be the number of $0$'s and $v_1$ be the number of $1$'s. Given such $u_0,u_1,v_0,v_1$, we have
\[\widetilde{T}=(t+1-k_0-k_1) \bar\eta_0+(u_0+v_0) \eta_0+(u_1+v_1)\eta_1, \]
and the frequency is
\[\binom{l_0+1}{k_0}\binom{l_1}{k_1}\binom{k_0}{u_0}\binom{k_1}{v_0}\Omega_{{\scriptsize \underbrace{0 \cdots 0}_{u_0+v_1} \underbrace{1\cdots 1}_{u_1+v_0}}}. \]
Now let $k$ and $u$ be fixed such that $k_0+k_1=k$ and $u_0+v_0=u$, where $0 \le u \le k_0+k_1=k$ and $2 \le k \le l_0+l_1+1=t+1$, we conclude that $\widetilde{T}$ takes the value
\begin{eqnarray} \label{4:t} \widetilde{T}=(t+1-k) \bar\eta_0+u \eta_0+(k-u)\eta_1, \end{eqnarray}
and the frequency is
\begin{eqnarray} \label{4:fre} \sum_{k_0=0}^k\sum_{u_0=0}^u\binom{l_0+1}{k_0}\binom{l_1}{k-k_0}
\binom{k_0}{u_0}\binom{k-k_0}{u-u_0}\Omega_{{\scriptsize \underbrace{0 \cdots 0}_{2u_0+k-k_0-u,} \underbrace{1\cdots 1}_{k_0+u-2u_0}}}. \end{eqnarray}
This, after returning to (\ref{equ-Weight of C}), provides the weight distribution of the cyclic code $\mathcal{C}_{(a_1,\ldots,a_t)}$ for the general case $N=2,t=e-1\geqslant 2$.

%Even though $\Omega_{{\scriptsize \underbrace{0 \cdots 0}_u \underbrace{1\cdots 1}_v}}$ can be evaluated explicitly for each $u,v$, the formula (\ref{4:fre}) is quite complicated and there seems no easy way to simplify the final form. So we would rather leave it as it is and calculate the weight distribution for the simplest case that $N=2, t=e-1=3$.

\subsection{Proof of Theorem \ref{thm-e2}}\label{sec-main2}
From $N=2=\gcd\left((q^m-1)/(q-1),4a\right)$ and $t=e-1=3$, it is easy to see that $q \equiv 1 \pmod{4}$, $m \equiv 2 \pmod{4}$ and $e=4\mid(q^{m/2}-1)$. If $2|a$, the weight distribution has been obtained from (i) of Theorem \ref{thm-e1} with $t=3$, this is Table \ref{Table3} in Theorem \ref{thm-e2}. If $2 \nmid a$, we use Theorem \ref{thm-e3} to calculate the weight distribution. In this case $l_0=1,l_1=2$, from (\ref{4:t}) and (\ref{4:fre}), for any $k,u$ with $2 \le k \le 4,\, 0 \le u \le k$, the sum $\widetilde{T}$ takes the value
\[\widetilde{T}=(4-k) \bar\eta_0+u \eta_0+(k-u)\eta_1, \]
with frequency
\[\sum_{k_0=0 }^k\sum_{u_0=0}^u\binom{2}{k_0}\binom{2}{k-k_0}
\binom{k_0}{u_0}\binom{k-k_0}{u-u_0}\Omega_{{\scriptsize \underbrace{0 \cdots 0}_{2u_0+k-k_0-u,} \underbrace{1\cdots 1}_{k_0+u-2u_0}}}. \]
Using the values
\begin{equation*}
  \left\{\begin{array}{l}
  \Omega_{00}=\Omega_{11}=\frac{r-1}{2}; \quad \Omega_{01}=0;\\
  \Omega_{000}=\Omega_{111}=\frac{r-1}{8}(r-5);\\
  \Omega_{001}=\Omega_{011}=\frac{(r-1)^2}{8};\\
  \Omega_{0000}=\Omega_{1111}=\frac{r-1}{16}(r^2-2r+9);\\
  \Omega_{0001}=\Omega_{0111} =\frac{r-1}{16}(r^2-4r+3);\\
  \Omega_{0011}=\frac{(r-1)^3}{16},
  \end{array}\right.
\end{equation*}
which we can obtain from Lemma \ref{thm-NNN3} in Section \ref{sec-app}, we find that for $k=2$,
$$\left.
    \begin{array}{ll}
      \mbox{Value} & \hbox{Frequency} \\
      2\bar\eta_0+2\eta_0, & \Omega_{00}+4\Omega_{01}+\Omega_{11}=r-1 \\[2mm]
      2\bar\eta_0+2\eta_1, & \Omega_{00}+4\Omega_{01}+\Omega_{11} =r-1 \\[2mm]
      2\bar\eta_0+\eta_0+\eta_{1}, & 4\Omega_{00}+4\Omega_{01}+4\Omega_{11} =4(r-1)\\[2mm]
    \end{array}
  \right.$$
and for $k=3$,
$$\left.
    \begin{array}{ll}
      \mbox{Value} & \hbox{Frequency} \\
      \bar\eta_0+3\eta_0, & 2\Omega_{001}+2\Omega_{011}=\frac{(r-1)^2}{2} \\[2mm]
      \bar\eta_0+3\eta_1, & 2\Omega_{001}+2\Omega_{011}=\frac{(r-1)^2}{2}\\[2mm]
      \bar\eta_0+2\eta_0+\eta_{1},&4\Omega_{011}+2\Omega_{000}+2\Omega_{111}+4\Omega_{001}\\
                                  &\quad =\frac{(r-1)}{2}(3r-7)\\[2mm]
      \bar\eta_0+\eta_0+2\eta_{1},&4\Omega_{011}+2\Omega_{000}+2\Omega_{111}+4\Omega_{001}\\
                                  &\quad =\frac{(r-1)}{2}(3r-7)
    \end{array}
  \right.$$
and for $k=4$,
$$\left.
    \begin{array}{ll}
      \mbox{Value} & \hbox{Frequency} \\
      4\eta_0, & \Omega_{0011}=\frac{(r-1)^3}{16} \\[2mm]
      4\eta_1, & \Omega_{0011}=\frac{(r-1)^3}{16} \\[2mm]
      3\eta_0+\eta_1, & 2\Omega_{0111}+2\Omega_{0001}=\frac{(r-1)}{4}(r^2-4r+3)  \\[2mm]
      \eta_0+3\eta_1, & 2\Omega_{0111}+2\Omega_{0001}=\frac{(r-1)}{4}(r^2-4r+3)  \\[2mm]
      2(\eta_0+\eta_{1}),&\Omega_{1111}+4\Omega_{0011}+\Omega_{0000}=\frac{(r-1)}{8}(3r^2-6r+11).
    \end{array}
  \right.$$

Now we have obtained the value distribution of $\widetilde{T}$. Returning to (\ref{equ-Tx}) and (\ref{equ-Weight of C}) gives us the weight distribution of the cyclic code $\mathcal{C}_{(a_1,a_2,a_3)}$ with $2 \nmid a$, which is summarized in Table \ref{Table4} in Theorem \ref{thm-e2}. This completes the proof of Theorem \ref{thm-e2}.

\section{Appendix: Calculation of $\Omega_{i_1\cdots i_u i_{u+1}}$}\label{sec-app}

Recall that for positive integer $u$ and any sequence $i_1,\cdots,i_{u},i_{u+1}\in\{0,1\}$, the value $\Omega_{i_1\cdots i_ui_{u+1}}$ is defined by
\begin{equation*}
 \Omega_{i_1\cdots i_ui_{u+1}}:=\#\left\{(x_1,\cdots ,x_u)\in (\mathbb{F}_{r}^*)^u\ \left|\ x_1\in C_{i_1}^{(2,r)},\cdots, x_u\in C_{i_u}^{(2,r)}, \sum\limits_{j=1}^ux_j\in C_{i_{u+1}}^{(2,r)}\right\}\right..
\end{equation*}

We first prove that the value of $\Omega_{i_1\cdots i_ti_{u+1}}$ is related to reduced quadratic Jacobi sums which were introduced in Section \ref{sec-pre} before.

\begin{lemma}\label{thm-NNN}
The number $\Omega_{i_1\cdots i_ui_{u+1}}$ defined above equals to
$$\frac{r-1}{2^{u+1}}\sum\limits_{0 \le v_2,\cdots,v_{u+1} \le 1} (-1)^{\sum\limits_{j=2}^{u+1}(i_1+i_j)v_j} \rho\left((-1)^{\sum\limits_{j=2}^{u}v_j}\right) J^*(\rho^{v_2},\cdots,\rho^{v_{u+1}}),$$
where $\rho$ is the quadratic multiplicative character of $\mathbb{F}_{r}$.
\end{lemma}
\begin{proof}
For $x\in \mathbb{F}_{r}^*$, let $\chi$ denote a multiplicative character of $\mathbb{F}_{r}$. It is easy to check that
\begin{equation}\label{equ-orthogonal2}
    \frac{1}{2}\sum\limits_{\chi^2=\varepsilon}\chi(x \gamma^i)=
\left\{
  \begin{array}{ll}
    1, & \hbox{if\ }x\in C_i^{(L,r)}; \\
    0, & \hbox{otherwise.}
  \end{array}
\right.
\end{equation}
Suppose $\chi_1,\chi_2,\cdots ,\chi_{u+1}$ denote multiplicative characters of $\mathbb{F}_{r}$. By the relation (\ref{equ-orthogonal2}), we have
$$\begin{array}{l}
 \Omega_{i_1\cdots i_ui_{u+1}}\\
=\sum\limits_{x_1,\cdots x_{u}\in\mathbb{F}_{r}^*}\left[\frac{1}{2}\sum\limits_{\chi_1^2=\varepsilon}\chi_1(x_1 \gamma^{i_1})\right]
\cdots
\left[\frac{1}{2}\sum\limits_{\chi_u^2=\varepsilon}\chi_u(x_u \gamma^{i_u})\right]
\left[\frac{1}{2}\sum\limits_{\chi_{u+1}^2=\varepsilon}\chi_{u+1}(\gamma^{i_{u+1}}\sum\limits_{j=1}^{u}x_j)\right].
\end{array}$$
Expanding the right hand side and changing the order of summation we obtain
$$\begin{array}{l}
\frac{1}{2^{u+1}}\sum\limits_{\chi_j^2=\varepsilon\atop j=1,\cdots,u+1}\chi_1(\gamma^{i_1})\cdots\chi_{u}(\gamma^{i_{u}})\chi_{u+1}(\gamma^{i_{u+1}})
\\
\qquad \cdot\sum\limits_{x_1, \ldots,x_u\in \mathbb{F}_{r}^*}\chi_1(x_1)\cdots\chi_{u}(x_u)
\chi_{u+1}(x_1+x_2+\cdots+x_u), \end{array}$$
which gives
$$\begin{array}{l}
\frac{1}{2^{u+1}}\sum\limits_{\chi_j^2=\varepsilon\atop j=1,\cdots,u+1}\chi_1(\gamma^{i_1})\chi_2(-\gamma^{i_2})\cdots\chi_{u}(-\gamma^{i_{u}})\chi_{u+1}(\gamma^{i_{u+1}})
\\
\qquad \cdot\sum\limits_{x_1, \ldots,x_u\in \mathbb{F}_{r}^*}\chi_1\chi_2 \cdots \chi_{u+1}(x_1) \chi_2(x_2)\cdots\chi_{u}(x_u)
\chi_{u+1}(1-x_2-\cdots-x_u). \end{array}$$
This is
$$\begin{array}{l}
\frac{r-1}{2^{u+1}}\sum\limits_{\chi_j^2=\varepsilon\atop j=2,\cdots,u+1}\chi_2(-\gamma^{i_1+i_2})\cdots\chi_u(-\gamma^{i_1+i_u})
\chi_{u+1}(\gamma^{i_1+i_{u+1}})\\
\qquad \cdot\sum\limits_{x_2,\cdots,x_{u}\in \mathbb{F}_{r}^*}\chi_2(x_2)\cdots\chi_u(x_u)\chi_{u+1}(1-x_1-\cdots-x_u).\end{array}$$
So we obtain
$$\begin{array}{l}
\Omega_{i_1\cdots i_ui_{u+1}}=\frac{r-1}{2^{u+1}}\sum\limits_{0 \le v_2,\cdots,v_{u+1} \le 1}(-1)^{\sum\limits_{j=2}^{u}(i_1+i_j)v_j} \rho\left((-1)^{\sum\limits_{j=2}^{u}v_j}\right)J^*(\rho^{v_2},\cdots,\rho^{v_{u+1}}).
\end{array}$$
This completes the proof of Lemma \ref{thm-NNN}.
\end{proof}

\begin{lemma}\label{thm-NNN2}
Suppose that $-1$ is a square in $\gf_r$, then
\[
\Omega_{i_1\cdots i_ui_{u+1}}=\frac{r-1}{2^{u+1}}\left\{\frac{1}{r}\bigg((r-1)^u-(-1)^u\bigg)-(-1)^u \sum_{1 \le l \le \frac{u+1}{2} }r^{l-1}\sum_{1 \le j_1 <j_2<\cdots<j_{2l} \le u+1} (-1)^{\sum\limits_{k=1}^{2l}i_{j_k}} \right\} .
\]
\end{lemma}
\begin{proof}
Since $N=\gcd(\frac{q^m-1}{q-1},ea)=2$ implies $2|m$, then $-1=\gamma^{\frac{q^m-1}{2}}$ is always a square in this paper, from Lemma \ref{thm-NNN} we have
$$\begin{array}{l}
\Omega_{i_1\cdots i_ui_{u+1}}=\frac{r-1}{2^{u+1}}\sum\limits_{0 \le v_2,\cdots,v_{u+1} \le 1}(-1)^{\sum\limits_{j=2}^{u}(i_1+i_j)v_j} J^*(\rho^{v_2},\cdots,\rho^{v_{u+1}}).
\end{array}$$
Note that $J^*(\rho_1,,\cdots,\rho_u)$ does not depend on the order of the characters $\rho_1,\ldots,\rho_u$, so we have
$$\begin{array}{l}
\Omega_{i_1\cdots i_ui_{u+1}}=\frac{r-1}{2^{u+1}}\sum\limits_{I \subset \{2,\ldots,u+1\}}(-1)^{\sum\limits_{j \in I}(i_1+i_j)} J^*(\underbrace{\varepsilon,\ldots,\varepsilon}_{u-\#I},\underbrace{\rho,\cdots,\rho}_{\#I}).
\end{array}$$
Separating the cases that $I =\emptyset$, $\#I>0$ is even and $\#I$ is odd and applying Lemmas \ref{lem_J} and \ref{lem-J*}, we can obtain
$$\begin{array}{l}
\Omega_{i_1\cdots i_ui_{u+1}}=\frac{r-1}{2^{u+1}} \left\{A+B+C\right\},
\end{array}$$
where
\[A=J^*(\underbrace{\varepsilon,\cdots,\varepsilon}_{u})=\frac{1}{r}\bigg((r-1)^u-(-1)^u\bigg),\]
\[B=(-1)^{u+1}\sum\limits_{\substack{\emptyset \ne I \subset \{2,\ldots,u+1\}\\
\#I \mbox{\small \, is even}}}(-1)^{\sum\limits_{j \in I}i_j}r^{(\#I-2)/2}, \]
and
\[C=(-1)^{u+1}\sum\limits_{\substack{I \subset \{2,\ldots,u+1\}\\
\#I \mbox{\small \, is odd}}}(-1)^{i_1+\sum\limits_{j \in I}i_j}r^{(\#I-1)/2}.  \]
Setting $\#I=2l$ if $\#$ is even and $\#I=2l-1$ is $\#I$ is odd completes the proof of Lemma \ref{thm-NNN2}.
\end{proof}

It is easy to see from Lemma \ref{thm-NNN2} that the value $\Omega_{i_1\cdots i_ui_{u+1}}$ does not depend on the order of the sequence $i_1,\ldots,i_u,i_{u+1}$. Now we can prove

\begin{lemma}\label{thm-NNN3}
Suppose that $-1$ is a square in $\gf_r$, then
\[\Omega_{{\scriptsize \underbrace{0 \cdots 0}_u \underbrace{1\cdots 1}_v}}=\frac{r-1}{r 2^{u+v+1}}\bigg\{2(r-1)^{u+v-1}+(-1)^{u+v} \left\{(1+\sqrt{r})^u(1-\sqrt{r})^v+(1-\sqrt{r})^u(1+\sqrt{r})^v\right\}\bigg\}.\]
\end{lemma}
\begin{proof} From Lemma \ref{thm-NNN2}, it suffices to compute
\[P=\sum_{1 \le l \le \frac{u+v}{2} }r^{l-1}\sum_{1 \le j_1 <j_2<\cdots<j_{2l} \le u+v} (-1)^{\sum\limits_{k=1}^{2l}i_{j_k}}. \]

Since $i_j=0$ for $1 \le j \le u$ and $i_k=1$ for $u+1 \le j \le u+v$, we have \[\sum_{1 \le j_1 <j_2<\cdots<j_{2l} \le u+v} (-1)^{\sum\limits_{k=1}^{2l}i_{j_k}}= \sum_{s=0}^{2l}\binom{u}{2l-s} \binom{v}{s} (-1)^s, \]
and the right hand side is the coefficient of $x^{2l}$ in the expansion of the polynomial $f(x):=(1+x)^u(1-x)^v$. Hence letting
\[f(x)=1+\sum_{n=1}^{u+v}a_nx^n,\quad a_n \in \mathbb{R},\]
then
\[P=\frac{1}{r}\sum_{1 \le l \le \frac{u+v}{2}} a_{2l} (\sqrt{r})^{2l}.\]
Clearly the right hand side is
\[\frac{1}{r} \left\{\frac{f(\sqrt{r})+f(-\sqrt{r})}{2}-1\right\}. \]
This completes the proof of Lemma \ref{thm-NNN3}.
\end{proof}

\section{Conclusions}\label{sec-conclusion}
In this paper, we determine the weight distributions of a new family of cyclic codes with arbitrary number of zeros, more precisely the cyclic codes $\mathcal{C}_{(a_1,\cdots,a_t)}$ given by (\ref{def}) with any $t \ge 2$ zeros under the conditions that $t=e-1$ and $N=2$. Our main results are as follows:
%Except for the recent works \cite{gegeng,Y-X-D12}, most of previous results are obtained for cyclic codes with no more than three zeros.
\begin{itemize}
  \item For $N=2$, $t=e-1 \geqslant 2$, $2|a$ and $e|(q^{m/2}-1)$, we obtain the weight distribution of $\mathcal{C}_{(a_1,\cdots,a_t)}$.
  \item For $N=2$ and $t=e-1=3$, we obtain the weight distribution of $\mathcal{C}_{(a_1,\cdots,a_t)}$.
  \item For the general case of $N=2$ and $t=e-1\geqslant 2$, we present a computational formula to determine the weight distribution of $\mathcal{C}_{(a_1,\cdots,a_t)}$.
\end{itemize}

%For the case that $t=e-1\geqslant 3$ and $N\geqslant 3$, the method in this paper might work if the corresponding Gaussian periods and Jacobi sums are known. However, it is expected that the computation will be very complicated. We may leave it for future work.

Except for these cases (in \cite{Y-X-D12} and this paper), the weight distribution of the code $\mathcal{C}_{(a_1,\cdots,a_t)}$ is open in most cases when $t < e$. It would be good if some of these open cases can be settled.

\subsection*{Acknowledgments}

Maosheng Xiong's research is supported by the Hong Kong Research Grants Council under Grant Nos. 609513 and 606211. Jing Yang's research is partly supported by the National Natural Science Foundation of China (No. 11371011). Lingli Xia's research is partly supported by Beijing Natural Science Foundation(No. 1144012) and Science and Technology on Information Assurance Laboratory (No. KJ-13-005).

\end{document}